\documentclass[12pt]{article}
\usepackage{amsmath,amsthm,amssymb}
\usepackage{amssymb,latexsym}

\newtheorem{theorem}{Theorem}
\newtheorem{conjecture}{Conjecture}
\newtheorem{lemma}{Lemma}

\begin{document}

\baselineskip=17pt

%\AddToShipoutPictureBG*{
%  \AtPageUpperLeft{
%  \hspace{\paperwidth}
%   \raisebox{-4\baselineskip}{
%   \makebox[-180pt][r]{
%     This is the author's version of the paper. The final publication has appeared in}
%
%     \raisebox{-1.2\baselineskip}{
%      \makebox[-188pt][r]{Georgian Math. J., \textbf{29}, 3, (2022), 455 -- 470.}}}}}

\title{\bf On an equation by primes with one Linnik prime}

\author{\bf S. I. Dimitrov}

\date{}

\maketitle
\begin{abstract}
Let $[\, \cdot\,]$ be the floor function.
In this paper, we  prove that when $1<c<\frac{16559}{15276}$, then
every sufficiently large positive integer $N$ can be represented in the form
\begin{equation*}
N=[p^c_1]+[p^c_2]+[p^c_3]\,,
\end{equation*}
where $p_1,p_2,p_3$ are primes, such that $p_1=x^2 + y^2 +1$.\\
\quad\\
\textbf{Keywords}: Diophantine equation $\cdot$ Prime $\cdot$ Exponential sum $\cdot$ Asymptotic formula\\
\quad\\
{\bf  2020 Math.\ Subject Classification}: 11L07 $\cdot$ 11L20 $\cdot$  11P32
\end{abstract}

\section{Introduction and main result}
\indent

A turning point in analytic number theory is 1937, when Vinogradov \cite{Vinogradov1}
proved the ternary Goldbach problem.
He showed that every sufficiently large odd integer $N$ can be represented  in the form
\begin{equation*}
N=p_1+p_2+p_3,
\end{equation*}
where $p_1,\, p_2,\, p_3$ are prime numbers.

Helfgott \cite{Helfgott} recently showed that this is true for all odd $N\geq7$.

The enormous consequences of Vinogradov's \cite{Vinogradov2} method for estimating exponential sums over primes
find objective expression in the hundreds of articles on diophantine equations and inequalities
by primes. Source of detailed proof of Vinogradov's theorem, beginning with an historical
perspective along with an overview of essential lemmas and theorems,
can be found in monograph of Rassias \cite{Rassias}.
For almost a century, Vinogradov's three primes theorem has been proved many times with prime numbers of a special form.
Recent interesting results in this regard are for example \cite{Dimitrov2}, \cite{Maier-Rassias1}, \cite{Maier-Rassias2}, \cite{Maier-Rassias3}.
As an analogue of the ternary Goldbach problem, in 1995, Laporta and Tolev  \cite{Laporta-Tolev}
investigated the diophantine equation
\begin{equation}\label{1}
[p^c_1]+[p^c_2]+[p^c_3]=N\,,
\end{equation}
where $p_1,\, p_2,\, p_3$ are primes, $c > 1$ and  $N$ is positive integer.
For $1<c <\frac{17}{16}$ and $\varepsilon>0$ they proved that for the sum
\begin{equation*}
R(N)=\sum\limits_{[p^c_1]+[p^c_2]+[p^c_3]=N}\log p_1\log p_2\log p_3
\end{equation*}
the asymptotic formula \begin{equation*}
R(N)=\frac{\Gamma^3\left(1 + \frac{1}{c}\right)}{\Gamma\left(\frac{3}{c}\right)}N^{\frac{3}{c}-1}
+\mathcal{O}\Big(N^{\frac{3}{c}-1}\exp\big(-(\log N)^{\frac{1}{3}-\varepsilon}\big)\Big)
\end{equation*} holds.

Afterwards the result of Laporta and Tolev was sharpened by Kumchev and Nedeva \cite{Kumchev-Nedeva}
to $1<c <\frac{12}{11}$, by Zhai and Cao  \cite{Zhai-Cao} to $1<c<\frac{258}{235}$,
by Cai \cite{Cai} to $1<c <\frac{137}{119}$, by Zhang and Li \cite{Zhang-Li} to $1<c <\frac{3113}{2703}$
and finally by Baker \cite{Baker} to $1<c <\frac{3581}{3106}$ and this is the best result up to now.

On the  other hand in 1960 Linnik \cite{Linnik} showed that there exist infinitely many prime numbers of the form
$p=x^2 + y^2 +1$, where $x$ and $y$ are integers. More precisely he proved the asymptotic formula
\begin{equation*}
\sum_{p\leq X}r(p-1)=\pi\prod_{p>2}\bigg(1+\frac{\chi_4(p)}{p(p-1)}\bigg)\frac{X}{\log X}+
\mathcal{O}\bigg(\frac{X(\log\log X)^7}{(\log X)^{1+\theta_0}}\bigg)\,,
\end{equation*}
where $r(k)$ is the number of solutions of the
equation $k=x^2 + y^2$ in integers, $\chi_4(k)$ is the non-principal character modulo 4 and
\begin{equation}\label{theta0}
\theta_0=\frac{1}{2}-\frac{1}{4}e\log2=0.0289...
\end{equation}
Recently the author \cite{Dimitrov4} showed
that  for any fixed $1<c<\frac{427}{400}$,  every sufficiently large positive number $N$
and a small constant $\varepsilon>0$, the diophantine inequality
\begin{equation*}
|p_1^c+p_2^c+p_3^c-N|<\varepsilon
\end{equation*}
has a solution in primes $p_1,\,p_2,\,p_3$, such that $p_1=x^2 + y^2 +1$.

Given the last result, it is natural to expect that the diophantine equation \eqref{1}
has a solution in primes $p_1,\,p_2,\,p_3$, such that $p_1=x^2 + y^2 +1$.
Let $N$ is a sufficiently large positive integer and
\begin{equation}\label{X}
X=N^{\frac{1}{c}}\,.
\end{equation}
Define \begin{equation}\label{Gamma}
\Gamma= \sum\limits_{X/2<p_1,p_2,p_3\leq X\atop{[p^c_1]+[p^c_2]+[p^c_3]=N}}r(p_1-1)\log p_1\log p_2\log p_3\,.
\end{equation}
We establish the following theorem.
\begin{theorem}\label{Theorem} Let $1<c<\frac{16559}{15276}$.
Then for every sufficiently large positive integer $N$  the asymptotic formula
\begin{equation}\label{asymptoticformula1}
\Gamma=\pi\prod\limits_p \left(1+\frac{\chi_4(p)}{p(p-1)}\right)
\frac{\Gamma^3\left(1+\frac{1}{c}\right)}{\Gamma\left(\frac{3}{c}\right)}\left(1-\frac{1}{2^{3-c}}\right)N^{\frac{3}{c}-1}
+\mathcal{O}\Bigg(\frac{N^{\frac{3}{c}-1}  (\log\log N)^5}{(\log N)^{\theta_0}}\Bigg)
\end{equation}
holds.  Here $\theta_0$ is defined by \eqref{theta0}.
\end{theorem}
In addition we have the following challenge for the future.
\begin{conjecture}
There exists $c_0>1$  such that  for any fixed $1<c<c_0$,
and every sufficiently large positive integer $N$, the diophantine equation
\begin{equation*}
[p^c_1]+[p^c_2]+[p^c_3]=N\,,
\end{equation*}
has a solution in prime numbers $p_1,\,p_2,\,p_3$, such that
$p_1=x_1^2 + y_1^2 +1$, $p_2=x_2^2 + y_2^2 +1$, $p_3=x_3^2 + y_3^2 +1$.
\end{conjecture}

\section{Notations}
\indent

Let $N$ be a sufficiently large positive integer.
By $\varepsilon$ we denote an arbitrary small positive number, not the same in all appearances.
The letter $p$  with or without subscript will always denote prime number.
The notation $m\sim M$ means that $m$ runs through the interval $( M/2, M]$.
As usual $\varphi (n)$ is Euler's function and $\Lambda(n)$ is von Mangoldt's function.
We shall use the convention that a congruence, $m\equiv n\,\pmod {d}$ will be written as $m\equiv n\,(d)$.
Moreover $e(y)=e^{2\pi i y}$.
As usual $[t]$, $\{t\}$ and $\|t\|$ denote the integer part of $t$, the fractional part of $t$
and the distance from $t$ to the nearest integer, respectively.
We recall that $t=[t]+\{t\}$ and $\|t\|=\min(\{t\}_,1-\{t\})$.
We denote by $r(k)$ the number of solutions of the equation $k=x^2 + y^2$ in integers.
The symbol $\chi_4(k)$ will mean the non-principal character modulo 4.
Throughout this paper unless something else is said, we suppose that $1<c<\frac{16559}{15276}$.

Denote
\begin{align} \label{D}
&D=\frac{X^{\frac{1}{2}}}{(\log N)^{\frac{6A+34}{3}}}\,,\quad A>3\,;\\
\label{Delta}
&\Delta=X^{\frac{1}{4}-c}\,;\\
\label{H}
&H=X^{\frac{1283}{15276}}\,;\\
\label{SldalphaX}
&S_{l,d;J}(t)=\sum\limits_{p\in J\atop{p\equiv l\, (d)}} e(t [p^c])\log p\,;\\
\label{SalphaX}
&S(t)=S_{1,1;(X/2,X]}(t)\,;\\
\label{Sldoverline}
&\overline{S}_{l,d;J}(t)=\sum\limits_{p\in J\atop{p\equiv l\, (d)}} e(t p^c)\log p\,;\\
\label{Soverline}
&\overline{S}(t)=\overline{S}_{1,1;( X/2,X]}(t)\,;\\
\label{IJalphaX}
&I_J(t)=\int\limits_Je(t y^c)\,dy\,;\\
\label{IalphaX}
&I(t)=I_{(X/2,X]}(t)\,;\\
\label{Eytda}
&E(y,t,d,a)=\sum\limits_{\mu y<n\leq y\atop{n\equiv a\, (d)}}\Lambda(n)e(t n^c)
-\frac{1}{\varphi(d)}\int\limits_{\mu y}^{y}e(t x^c)\,dx\,,\\
&\mbox{ where } \quad 0<\mu<1\,.\nonumber
\end{align}

\section{Preliminary lemmas}
\indent

\begin{lemma}\label{Squareout}
For any complex numbers $a(l)$ we have
\begin{equation*}
\bigg|\sum_{L<l\le 2L}a(l)\bigg|^2
\leq\bigg(1+\frac{L}{Q}\bigg)\sum_{|q|\leq Q}\bigg(1-\frac{|q|}{Q}\bigg)
\sum_{L<l, \, l+q\leq 2L}a(l+q)\overline{a(l)},
\end{equation*}
where $Q\geq1$.
\end{lemma}
\begin{proof}
See (\cite{Heath}, Lemma 5).
\end{proof}

\begin{lemma}\label{Exponentpairs}
Let $|f^{(m)}(u)|\asymp YX^{1-m}$  for $1\leq X<u<X_0\leq2X$ and $m\geq1$.\\
Then
\begin{equation*}
\bigg|\sum_{X<n\le X_0}e(f(n))\bigg|
\ll Y^\varkappa X^\lambda +Y^{-1},
\end{equation*}
where $(\varkappa, \lambda)$ is any exponent pair.
\end{lemma}
\begin{proof}
See (\cite{Graham-Kolesnik}, Ch. 3).
\end{proof}

\begin{lemma}\label{Buriev} Let $x,y\in\mathbb{R}$ and $H\geq3$.
Then the formula
\begin{equation*}
e(-x\{y\})=\sum\limits_{|h|\leq H}c_h(x)e(hy)+\mathcal{O}\left(\min\left(1, \frac{1}{H\|y\|}\right)\right)
\end{equation*}
holds. Here
\begin{equation*}
c_h(x)=\frac{1-e(-x)}{2\pi i(h+x)}\,.
\end{equation*}
\end{lemma}
\begin{proof}
See (\cite{Buriev}, Lemma 12).
\end{proof}

\begin{lemma}\label{Heath-Brown} Let $3 < U < V < Z < X$ and suppose that $Z -\frac{1}{2}\in\mathbb{N}$,
$X\gg Z^2U$, $Z \gg U^2$, $V^3\gg X$.
Assume further that $F(n)$ is a complex valued function such that $|F(n)| \leq 1$.
Then the sum
\begin{equation*}
\sum\limits_{n\sim X}\Lambda(n)F(n)
\end{equation*}
can be decomposed into $O\Big(\log^{10}X\Big)$ sums, each of which is either of Type I
\begin{equation*}
\sum\limits_{m\sim M}a(m)\sum\limits_{l\sim L}F(ml)\,,
\end{equation*}
where
\begin{equation*}
L \gg Z\,, \quad  LM\asymp X\,, \quad |a(m)|\ll m^\varepsilon\,,
\end{equation*}
or of Type II
\begin{equation*}
\sum\limits_{m\sim M}a(m)\sum\limits_{l\sim L}b(l)F(ml)\,,
\end{equation*}
where
\begin{equation*}
U \ll L \ll V\,, \quad  LM\asymp X\,, \quad |a(m)|\ll m^\varepsilon\,,\quad |b(l)|\ll l^\varepsilon\,.
\end{equation*}
\end{lemma}
\begin{proof}
See (\cite{Heath}, Lemma 3).
\end{proof}

\begin{lemma}\label{SIasympt} Let $1<c<3$, $c\neq2$ and $|t|\leq\Delta$.
Then the asymptotic formula
\begin{equation*}
\sum\limits_{ X/2<p\leq X} e(t p^c)\log p=\int\limits_{X/2}^{X}e(t y^c)\,dy
+\mathcal{O}\left(\frac{X}{e^{(\log X)^{\frac{1}{5}}}}\right)
\end{equation*}
holds.
\end{lemma}
\begin{proof}
See (\cite{Tolev}, Lemma 14).
\end{proof}

\begin{lemma}\label{Bomb-Vin-Dim} Let $1<c<3$, $c\neq2$, $|t|\leq\Delta$ and $A>0$ be fixed.
Then the inequality
\begin{equation*}
\sum\limits_{d\le \sqrt{X}/(\log N)^{\frac{6A+34}{3}}}\max\limits_{y\le X}
\max\limits_{(a,\, d)=1}\big|E(y,t,d,a)\big|\ll\frac{X}{\log^AX}
\end{equation*}
holds. Here $\Delta$ and $E(y,t,d,a)$ are denoted by \eqref{Delta} and \eqref{Eytda}.
\end{lemma}
\begin{proof}
See (\cite{Dimitrov4}, Lemma 18).
\end{proof}

\begin{lemma}\label{intLintI}
For the sum denoted by \eqref{SalphaX} and the integral denoted by \eqref{IalphaX} we have
\begin{align*}
&\emph{(i)}\quad\quad\quad\;\,
\int\limits_{-\Delta}^\Delta|S(t)|^2\,dt\,\ll X^{2-c}\log^2X\,,
\quad\quad\quad\quad\quad\quad\quad\\
&\emph{(ii)}\quad\quad\quad\int\limits_{-\Delta}^\Delta|I(t)|^2\,dt\ll X^{2-c}\log X\,,\\
\quad\quad\quad\quad\quad\quad\quad
&\emph{(iii)}\quad\quad\;\,
\int\limits_{0}^{1}|S(t)|^2\,dt\ll X\log X\,.
\quad\quad\quad\quad\quad\quad\quad
\end{align*}
\end{lemma}
\begin{proof}
It follows from the arguments used in (\cite{Tolev}, Lemma 7).
\end{proof}

\begin{lemma}\label{intSld}
For the sum denoted by \eqref{SldalphaX} we have
\begin{equation*}
\int\limits_{-\Delta}^\Delta|S_{l,d;J}(t)|^2\,dt\ll\frac{X^{2-c}\log^3X}{d^2}\,.
\end{equation*}
\end{lemma}
\begin{proof}
It follows by the arguments used in (\cite{Dimitrov1}, Lemma 6 (i)).
\end{proof}

\begin{lemma}\label{Wuest} Let  $\alpha$, $\beta$ be real numbers such that
\begin{equation*}
\alpha\beta(\alpha-1)(\beta-1)\neq0\,.
\end{equation*}
Set
\begin{equation*}
\Sigma_{I}=\sum\limits_{m\sim M}a(m)
\sum\limits_{l\in I_m}e\left(F\frac{m^\alpha l^\beta}{M^\alpha L^\beta}\right)\,,
\end{equation*}
where
\begin{equation*}
F>0\,, \quad M\geq1\,, \quad L\geq1\,, \quad |a_m|\ll1
\end{equation*}
and $I_m$ is a subinterval of $( L/2, L]$.
Then for any exponent pair $(\varkappa, \lambda)$, we have
\begin{equation*}
\Sigma_{I}\ll \left(F^{\frac{1+2\varkappa}{6+4\varkappa}}M^{\frac{4+4\varkappa}{6+4\varkappa}}L^{\frac{3+2\lambda}{6+4\varkappa}}
+M^{\frac{1}{2}}L+ML^{\frac{1}{2}}+F^{-1}ML\right)\log(2+FML)\,.
\end{equation*}
\end{lemma}
\begin{proof}
See (\cite{Wu}, Theorem  2).
\end{proof}

The next two lemmas are due to C. Hooley.
\begin{lemma}\label{Hooley1}
For any constant $\omega>0$ we have
\begin{equation*}
\sum\limits_{p\leq X}
\bigg|\sum\limits_{d|p-1\atop{\sqrt{X}(\log X)^{-\omega}<d<\sqrt{X}(\log X)^{\omega}}}
\chi_4(d)\bigg|^2\ll \frac{X(\log\log X)^7}{\log X}\,,
\end{equation*}
where the constant in Vinogradov's symbol depends on $\omega>0$.
\end{lemma}

\begin{lemma}\label{Hooley2} Suppose that $\omega>0$ is a constant
and let $\mathcal{F}_\omega(X)$ be the number of primes $p\leq X$
such that $p-1$ has a divisor in the interval $\big(\sqrt{X}(\log X)^{-\omega}, \sqrt{X}(\log X)^\omega\big)$.
Then
\begin{equation*}
\mathcal{F}_\omega(X)\ll\frac{X(\log\log X)^3}{(\log X)^{1+2\theta_0}}\,,
\end{equation*}
where $\theta_0$ is defined by \eqref{theta0} and the constant in
Vinogradov's symbol depends only on $\omega>0$.
\end{lemma}
The proofs of very similar results are available in (\cite{Hooley}, Ch.5).

\section{Outline of the proof}
\indent
From \eqref{Gamma} and well-known identity
\begin{equation*}
r(n)=4\sum_{d|n}\chi_4(d)
\end{equation*}
we obtain
\begin{equation} \label{Gamma0decomp}
\Gamma=4\big(\Gamma_1+\Gamma_2+\Gamma_3\big),
\end{equation}
where
\begin{align}
\label{Gamma1}
&\Gamma_1=\sum\limits_{X/2<p_1,p_2,p_3\leq X\atop{[p^c_1]+[p^c_2]+[p^c_3]=N}}
\left(\sum\limits_{d|p_1-1\atop{d\leq D}}\chi_4(d)\right)\log p_1\log p_2\log p_3\,,\\
\label{Gamma2}
&\Gamma_2=\sum\limits_{X/2<p_1,p_2,p_3\leq X\atop{[p^c_1]+[p^c_2]+[p^c_3]=N}}
\left(\sum\limits_{d|p_1-1\atop{D<d<X/D}}\chi_4(d)\right)\log p_1\log p_2\log p_3\,,\\
\label{Gamma3}
&\Gamma_3=\sum\limits_{X/2<p_1,p_2,p_3\leq X\atop{[p^c_1]+[p^c_2]+[p^c_3]=N}}
\left(\sum\limits_{d|p_1-1\atop{d\geq X/D}}\chi_4(d)\right)\log p_1\log p_2\log p_3\,.
\end{align}
In order to estimate $\Gamma_1$ and $\Gamma_3$ we have to consider
the sum
\begin{equation} \label{Ild}
I_{l,d;J}(N)=\sum\limits_{X/2<p_2,p_3\leq X\atop{[p^c_1]+[p^c_2]+[p^c_3]=N\atop{p_1\equiv l\;(d)
\atop{p_1\in J}}}}\log p_1\log p_2\log p_3\,,
\end{equation}
where $d$ and $l$ are coprime natural numbers, and $J\subset(X/2,X]$-interval.
If $J=(X/2,X]$ then we write for simplicity $I_{l,d}(N)$.
Clearly

\begin{align}\label{Ilddecomp}
I_{l,d;J}(N)&=\int\limits_{-\Delta}^{1-\Delta}S_{l,d;J}(t)S^2(t)e(-t N)\,dt\nonumber\\
&=\int\limits_{-\Delta}^{\Delta}S_{l,d;J}(t)S^2(t)e(-t N)\,dt
+\int\limits_{\Delta}^{1-\Delta}S_{l,d;J}(t)S^2(t)e(-t N)\,dt\nonumber\\
&=I_{l,d;J}^{(1)}(N)+I_{l,d;J}^{(2)}(N)\,,
\end{align}
where
\begin{align}
\label{Ild1}
&I_{l,d;J}^{(1)}(N)=\int\limits_{-\Delta}^{\Delta}S_{l,d;J}(t)S^2(t)e(-t N)\,dt\,,\\
\label{Ild2}
&I_{l,d;J}^{(2)}(N)=\int\limits_{\Delta}^{1-\Delta}S_{l,d;J}(t)S^2(t)e(-t N)\,dt\,.
\end{align}
We shall estimate $I^{(1)}_{l,d;J}(N)$, $\Gamma_3$, $\Gamma_2$ and $\Gamma_1$,
respectively, in the sections \ref{SectionIld1}, \ref{SectionGamma3},
\ref{SectionGamma2} and \ref{SectionGamma1}.
In section \ref{Sectionfinal} we shall finalize the proof of Theorem \ref{Theorem}.

\section{Asymptotic formula for $\mathbf{I^{(1)}_{l,d;J}(N)}$}\label{SectionIld1}
\indent

Using \eqref{SldalphaX}, \eqref{Sldoverline} and $|t|\leq\Delta$ we write
\begin{align}\label{Sld-Soverline}
S_{l,d;J}(t)&=\sum\limits_{p\in J\atop{p\equiv l\, (d)}} e\big(t p^c+\mathcal{O}(|t|)\big)\log p
=\sum\limits_{p\in J\atop{p\equiv l\, (d)}} e(t p^c)\big(1+\mathcal{O}(|t|)\big)\log p\nonumber\\
&= \overline{S}_{l,d;J}(t)+\mathcal{O}\left(\frac{\Delta X\log X}{d}\right)\,.
\end{align}
Put
\begin{align}
\label{S1}
&S_1=S(t)\,,\\
\label{S2}
&S_2=S_{l,d;J}(t)\,,\\
\label{I1}
&I_1=I(t)\,,\\
\label{I2}
&I_2=\frac{I_J(t)}{\varphi(d)}\,.
\end{align}
We use the identity
\begin{equation}\label{Identity}
S^2_1S_2=I^2_1I_2+(S_2-I_2)I_1^2+S_2(S_1-I_1)I_1+S_1S_2(S_1-I_1)\,.
\end{equation}
Define
\begin{equation}\label{PhiDeltaJX}
\Phi_{\Delta, J}(X, d)=\frac{1}{\varphi(d)}
\int\limits_{-\Delta}^{\Delta}I^2(t)I_J(t)e(-Nt)\,dt\,.
\end{equation}
From  \eqref{Delta} -- \eqref{IalphaX}, \eqref{Ild1},
\eqref{Sld-Soverline} -- \eqref{PhiDeltaJX}, Lemma \ref{SIasympt},
Lemma \ref{intLintI}, Lemma \ref{intSld} and  Cauchy's inequality it follows
\begin{align}\label{Ild1-PhiDeltaJX}
I^{(1)}_{l,d;J}(X)-\Phi_{\Delta, J}(X, d)
&=\int\limits_{-\Delta}^{\Delta}\Bigg(S_{l,d;J}(t)-\frac{I_J(t)}{\varphi(d)}\Bigg)
I^2(t)e(-N t)\,dt\nonumber\\
&+\int\limits_{-\Delta}^{\Delta}S_{l,d;J}(t)\Big(S(t)-I(t)\Big)I(t)e(-Nt)\,dt\nonumber\\
&+\int\limits_{-\Delta}^{\Delta}S(t)S_{l,d;J}(t)\Big(S(t)-I(t)\Big)e(-Nt)\,dt\nonumber\\
&\ll \Bigg(\max\limits_{|t|\leq\Delta}\bigg|\overline{S}_{l,d;J}(t)-\frac{I_J(t)}{\varphi(d)}\bigg|+\frac{\Delta X\log X}{d}\Bigg)
\int\limits_{-\Delta}^{\Delta}|I(t)|^2\,dt\nonumber\\
&+\Bigg(\frac{X}{e^{(\log X)^{\frac{1}{5}}}}+\Delta X\Bigg)
\Bigg(\int\limits_{-\Delta}^{\Delta}|S_{l,d;J}(t)|^2\,dt\Bigg)^{\frac{1}{2}}
\Bigg(\int\limits_{-\Delta}^{\Delta}|I(t)|^2\,dt)\Bigg)^{\frac{1}{2}}\nonumber\\
&+\Bigg(\frac{X}{e^{(\log X)^{\frac{1}{5}}}}+\Delta X\Bigg)
\Bigg(\int\limits_{-\Delta}^{\Delta}|S(t)|^2\,dt\Bigg)^{\frac{1}{2}}
\Bigg(\int\limits_{-\Delta}^{\Delta}|S_{l,d;J}(t)|^2\,dt\Bigg)^{\frac{1}{2}}\nonumber\\
&\ll X^{2-c}(\log X)\max\limits_{|t|\leq\Delta}
\bigg|\overline{S}_{l,d;J}(t)-\frac{I_J(t)}{\varphi(d)}\bigg|
+\frac{X^{3-c}}{de^{(\log X)^{\frac{1}{6}}}}\,.
\end{align}
Put
\begin{equation}\label{PhiJX}
\Phi_J(X, d)=\frac{1}{\varphi(d)}\int\limits_{-\infty}^{\infty}I^2(t)I_J(t)e(-Nt)\,dt\,.
\end{equation}
Using \eqref{IJalphaX}, \eqref{IalphaX}, \eqref{PhiDeltaJX}, \eqref{PhiJX} and the estimations
\begin{equation}\label{IalphaXest}
I_J(t)\ll \min \left(X, \,\frac{X^{1-c}}{|t|}\right)\,,  \quad
I(t)\ll \min \left(X, \,\frac{X^{1-c}}{|t|}\right)
\end{equation}
we deduce
\begin{equation*}
\Phi_{\Delta, J}(X, d)-\Phi_J(X, d)\ll \frac{1}{\varphi(d)}
\int\limits_{\Delta}^{\infty} |I(t)|^2|I_J(t)|\,dt
\ll\frac{ X^{3-3c}}{\varphi(d)} \int\limits_{\Delta}^{\infty} \frac{dt}{t^3}
\ll \frac{ X^{3-3c}}{\varphi(d)\Delta^2}
\end{equation*}
and therefore
\begin{equation}\label{JXest1}
\Phi_{\Delta, J}(X, d)=\Phi_J(X, d)
+\mathcal{O}\left(\frac{ X^{3-3c}}{\varphi(d)\Delta^2}\right)\,.
\end{equation}
Finally  \eqref{Delta}, \eqref{Ild1-PhiDeltaJX}, \eqref{JXest1} and the identity
\begin{equation*}
I^{(1)}_{l,d;J}(X)=I^{(1)}_{l,d;J}(X)-\Phi_{\Delta, J}(X, d)
+\Phi_{\Delta, J}(X, d)-\Phi_J(X, d)+\Phi_J(X, d)
\end{equation*}
yield
\begin{equation}\label{Ild1est}
I^{(1)}_{l,d;J}(X)=\Phi_J(X, d)
+\mathcal{O}\Bigg( X^{2-c}
(\log X)\max\limits_{|t|\leq\Delta}\bigg|\overline{S}_{l,d;J}(t)-\frac{I_J(t)}{\varphi(d)}\bigg|\Bigg)
+\mathcal{O}\bigg(\frac{ X^{3-c}}{de^{(\log X)^{\frac{1}{6}}}}\bigg)\Bigg)\,.
\end{equation}

We are now in a good position to  estimate the sum $\Gamma_3$.

\section{Upper bound of $\mathbf{\Gamma_3}$}\label{SectionGamma3}
\indent

Consider the sum  $\Gamma_3$.\\
Since
\begin{equation*}
\sum\limits_{d|p_1-1\atop{d\geq X/D}}\chi_4(d)=\sum\limits_{m|p_1-1\atop{m\leq (p_1-1)D/X}}
\chi_4\bigg(\frac{p_1-1}{m}\bigg)
=\sum\limits_{j=\pm1}\chi_4(j)\sum\limits_{m|p_1-1\atop{m\leq (p_1-1)D/X
\atop{\frac{p_1-1}{m}\equiv j\,(4)}}}1
\end{equation*}
then from \eqref{Gamma3} and \eqref{Ild} we get
\begin{equation*}
\Gamma_3=\sum\limits_{m<D\atop{2|m}}\sum\limits_{j=\pm1}\chi_4(j)I_{1+jm,4m;J_m}(N)\,,
\end{equation*}
where $J_m=\big(\max\{1+mX/D,X/2\},X\big]$.
The last formula and \eqref{Ilddecomp} imply
\begin{equation}\label{Gamma3decomp}
\Gamma_3=\Gamma_3^{(1)}+\Gamma_3^{(2)}\,,
\end{equation}
where
\begin{equation}\label{Gamma3i}
\Gamma_3^{(i)}=\sum\limits_{m<D\atop{2|m}}\sum\limits_{j=\pm1}\chi_4(j)
I_{1+jm,4m;J_m}^{(i)}(N)\,,\;\; i=1,\,2\,.
\end{equation}

\subsection{Estimation of $\mathbf{\Gamma_3^{(1)}}$}
\indent

From \eqref{Ild1est} and \eqref{Gamma3i} we obtain
\begin{align}\label{Gamma31}
\Gamma_3^{(1)}=\Gamma^*&
+\mathcal{O}\Big( X^{2-c}(\log X)\Sigma_1\Big)
+\mathcal{O}\bigg(\frac{ X^{3-c}}{e^{(\log X)^{\frac{1}{6}}}}\Sigma_2\bigg)\,,
\end{align}
where
\begin{align}
\label{Gamma*}
&\Gamma^*=\sum\limits_{m<D\atop{2|m}}
\Phi_J(X, 4m)\sum\limits_{j=\pm1}\chi_4(j)\,,\\
\label{Sigma1}
&\Sigma_1=\sum\limits_{m<D\atop{2|m}}
\max\limits_{|t|\leq\Delta}\bigg|\overline{S}_{1+jm,4m;J}(t)-\frac{I_J(t)}{\varphi(4m)}\bigg|\,,\\
\label{Sigma2}
&\Sigma_2=\sum\limits_{m<D}\frac{1}{4m}\,.
\end{align}
From the properties of $\chi(k)$ we have that
\begin{equation}\label{Gamma*est}
\Gamma^*=0\,.
\end{equation}
By \eqref{D}, \eqref{SldalphaX},  \eqref{IJalphaX}, \eqref{Sigma1} and Lemma \ref{Bomb-Vin-Dim} we find
\begin{equation}\label{Sigma1est}
\Sigma_1\ll\frac{X}{\log^AX}\,.
\end{equation}
It is well known that
\begin{equation}\label{Sigma2est}
\Sigma_2\ll \log X\,.
\end{equation}
Bearing in mind \eqref{Gamma31}, \eqref{Gamma*est}, \eqref{Sigma1est} and \eqref{Sigma2est} we deduce
\begin{equation}\label{Gamma31est}
\Gamma_3^{(1)}\ll\frac{ X^{3-c}}{\log X}\,.
\end{equation}

\subsection{Estimation of $\mathbf{\Gamma_3^{(2)}}$}
\indent

Now we consider $\Gamma_3^{(2)}$. The formulas \eqref{Ild2} and \eqref{Gamma3i} give us
\begin{equation}\label{Gamma32}
\Gamma_3^{(2)}=\int\limits_{\Delta}^{1-\Delta}S^2(t)K(t)e(-Nt)\,dt\,,
\end{equation}
where
\begin{equation}\label{Kt}
K(t)=\sum\limits_{m<D\atop{2|m}}\sum\limits_{j=\pm1}\chi_4(j)S_{1+jm,4m;J_m}(t)\,.
\end{equation}

\begin{lemma}\label{IntK2}
For the sum denoted by \eqref{Kt} we have
\begin{equation*}
\int\limits_{0}^{1}|K(t)|^2\,dt\ll X\log^6X\,.
\end{equation*}
\end{lemma}
\begin{proof}
See (\cite{Dimitrov3}, Lemma 22).
\end{proof}

\begin{lemma}\label{SIest} Assume that
\begin{equation}\label{Conditions1}
\Delta \leq |t| \leq 1-\Delta\,, \quad |a(m)|\ll m^\varepsilon \,,\quad LM\asymp X\,,\quad L\gg X^{\frac{4}{9}}
\end{equation}
and $c_h(t)$ denote complex numbers such that $|c_h(t)|\ll (1+|h|)^{-1}$.\\
Set
\begin{equation*}
S_I=\sum\limits_{|h|\leq H}c_h(t)\sum\limits_{m\sim M}a(m)\sum\limits_{l\sim L}e((h+t)m^cl^c)\,.
\end{equation*}
Then
\begin{equation*}
S_I\ll X^{\frac{13993}{15276}+\varepsilon}\,.
\end{equation*}
\end{lemma}

\begin{proof} We have

\begin{equation}\label{SI}
S_I\ll X^\varepsilon\max\limits_{|\eta|\in (\Delta, H+1)}  \sum\limits_{m\sim M}\left|\sum\limits_{l\sim L}e(\eta m^cl^c)\right|\,.
\end{equation}
We first consider the case when
\begin{equation}\label{M411}
M\ll X^{\frac{12620}{34371}}\,.
\end{equation}
From \eqref{Delta}, \eqref{H}, \eqref{Conditions1}, \eqref{SI}, \eqref{M411} and Lemma \ref{Exponentpairs}
with the exponent pair $\left(\frac{2}{40},\frac{33}{40}\right)$  it follows
\begin{align}\label{SIest1}
S_I&\ll  X^\varepsilon\max\limits_{|\eta|\in (\Delta, H+1)}
\sum\limits_{m\sim M}\Bigg(\big(|\eta|X^cL^{-1}\big)^{\frac{2}{40}}L^{\frac{33}{40}}+\frac{1}{|\eta|X^cL^{-1}}\Bigg)\nonumber\\
&\ll  X^\varepsilon\max\limits_{|\eta|\in (\Delta, H+1)}\Bigg(|\eta|^{\frac{2}{40}}X^{\frac{2c}{40}}ML^{\frac{31}{40}}
+\frac{LM}{|\eta|X^c} \Bigg)\nonumber\\
&\ll  X^\varepsilon\Big(H^{\frac{1}{20}}X^{\frac{2c+31}{40}}M^{\frac{9}{40}}+   \Delta^{-1}X^{1-c}\Big)\nonumber\\
&\ll X^{\frac{13993}{15276}+\varepsilon}\,.
\end{align}
Next we consider the case when
\begin{equation}\label{M41135}
X^{\frac{12620}{34371}}\ll M\ll X^{\frac{5}{9}}\,.
\end{equation}
Using \eqref{SI}, \eqref{M41135} and Lemma \ref{Wuest} with the exponent pair $\left(\frac{2}{7},\frac{4}{7}\right)$ we deduce
\begin{align}\label{SIest2}
S_I&\ll X^\varepsilon\max\limits_{|\eta|\in (\Delta, H+1)}
\bigg((|\eta X^c|)^{\frac{11}{50}}M^{\frac{36}{50}}L^{\frac{29}{50}}+M^{\frac{1}{2}}L+ML^{\frac{1}{2}}+|\eta|^{-1}X^{-c}LM\bigg)\nonumber\\
&\ll X^\varepsilon
\Big(H^{\frac{11}{50}}M^{\frac{7}{50}}X^{\frac{11c+29}{50}}+XM^{-\frac{1}{2}}+X^{\frac{1}{2}}M^{\frac{1}{2}}+\Delta^{-1}X^{1-c}\Big)\nonumber\\
&\ll X^{\frac{13993}{15276}+\varepsilon}\,.
\end{align}
Bearing in mind \eqref{SIest1} and \eqref{SIest2} we establish the statement in the lemma.
\end{proof}

\begin{lemma}\label{SIIest} Assume that
\begin{equation}\label{Conditions2}
\Delta \leq |t| \leq 1-\Delta\,, \quad |a(m)|\ll m^\varepsilon \,, \quad |b(l)|\ll l^\varepsilon\,,
\quad LM\asymp X\,,\quad X^{\frac{1}{9}} \ll L\ll X^{\frac{1}{3}}
\end{equation}
and $c_h(t)$ denote complex numbers such that $|c_h(t)|\ll (1+|h|)^{-1}$.\\
Set
\begin{equation*}
S_{II}=\sum\limits_{|h|\leq H}c_h(t)\sum\limits_{m\sim M}a(m)\sum\limits_{l\sim L}b(l)e((h+t)m^cl^c)\,.
\end{equation*}
Then
\begin{equation*}
S_{II}\ll X^{\frac{13993}{15276}+\varepsilon}\,.
\end{equation*}
\end{lemma}
\begin{proof}

Using Cauchy's inequality and Lemma \ref{Squareout} with $Q=X^{\frac{3839}{15276}}$ we obtain
\begin{equation}\label{SIIest1}
|S_{II}|\ll X^\varepsilon\sum\limits_{|h|\leq H}\big|c_h(t)\big|\Bigg(\frac{X^2}{Q}+\frac{X}{Q}\sum\limits_{1\leq q\leq Q}
\sum\limits_{l\sim L}\bigg|\sum\limits_{m\sim M}e\big(f(l, m, q)\big)\bigg|\Bigg)^{\frac{1}{2}}\,,
\end{equation}
where $f_h(l, m, q)=(h+t)m^c\big((l+q)^c-l^c\big)$.
Now  \eqref{H}, \eqref{Conditions2}, \eqref{SIIest1} and Lemma \ref{Exponentpairs} with the exponent pair
\begin{equation*}
(\varkappa, \lambda)=BABABA^2BA^3BA^2B(0, 1)=\left(\frac{214}{845},\frac{199}{338}\right)
\end{equation*}
imply
\begin{align*}
S_{II}&\ll X^\varepsilon\sum\limits_{|h|\leq H}\big|c_h(t)\big|\Bigg(\frac{X^2}{Q}+\frac{X}{Q}\sum\limits_{1\leq q\leq Q}
\sum\limits_{l\sim L}\bigg(\big(|h+t|qX^{c-1}\big)^{\frac{214}{845}}M^{\frac{199}{338}}+\frac{1}{|h+t|qX^{c-1}}\bigg)\Bigg)^{\frac{1}{2}}\\
&\ll X^\varepsilon\sum\limits_{|h|\leq H}\big|c_h(t)\big|\Bigg(\frac{X^2}{Q}+\frac{X}{Q}
\bigg(H^{\frac{214}{845}}X^{\frac{214(c-1)}{845}}M^{\frac{199}{338}}Q^{\frac{1059}{845}}L+\Delta^{-1}X^{1-c}L\log Q\bigg)\Bigg)^{\frac{1}{2}}\\
&\ll X^{\frac{13993}{15276}+\varepsilon}\sum\limits_{|h|\leq H}\big|c_h(t)\big|
\ll X^{\frac{13993}{15276}+\varepsilon}\sum\limits_{|h|\leq H}\frac{1}{1+|h|}
\ll X^{\frac{13993}{15276}+\varepsilon}\,,
\end{align*}
which proves the statement in the lemma.
\end{proof}

\begin{lemma}\label{SalphaXest} Let $\Delta \leq |t| \leq 1-\Delta$.
Then  for the exponential sum denoted by \eqref{SalphaX} we have
\begin{equation*}
S(t)\ll X^{\frac{13993}{15276}+\varepsilon}\,.
\end{equation*}
\end{lemma}
\begin{proof}
In order to prove the lemma  we will use the formula
\begin{equation}\label{Lambdalog2}
S(t)=S^\ast(t)+\mathcal{O}\big(X^{\frac{1}{2}}\big)\,,
\end{equation}
where
\begin{equation}\label{Sast}
S^\ast(t)=\sum\limits_{X/2<n\leq X}\Lambda(n)e(t [n^c])\,.
\end{equation}
By \eqref{H}, \eqref{Sast} and Lemma \ref{Buriev} with $x=t$ and $y=n^c$ we get
\begin{align}\label{Sastformula}
S^\ast(t)&=\sum\limits_{X/2<n\leq X}\Lambda(n)e(t n^c-t\{n^c\})=\sum\limits_{X/2<n\leq X}\Lambda(n)e(t n^c)e(-t\{n^c\})\nonumber\\
&=\sum\limits_{X/2<n\leq X}\Lambda(n)e(tn^c)\left(\sum\limits_{|h|\leq H}c_h(t)e(hn^c)
+\mathcal{O}\Bigg(\min\left(1, \frac{1}{H\|n^c\|}\right)\Bigg)\right)\nonumber\\
&=\sum\limits_{|h|\leq H}c_h(t)\sum\limits_{X/2<n\leq X}\Lambda(n)e((h+t)n^c)
+\mathcal{O}\Bigg((\log X)\sum\limits_{X/2<n\leq X}\min\left(1, \frac{1}{H\|n^c\|}\right)\Bigg)\nonumber\\
&=S_0^\ast(t)+\mathcal{O}\left((\log X)\sum\limits_{X/2<n\leq X}\min\left(1, \frac{1}{H\|n^c\|}\right)\right)\,,
\end{align}
where
\begin{equation}\label{Sast0}
S_0^\ast(t)=\sum\limits_{|h|\leq H}c_h(t)\sum\limits_{X/2<n\leq X}\Lambda(n)e((h+t)n^c)\,.
\end{equation}
Arguing as in (\cite{Cai}, Lemma 3.3) we find
\begin{equation}\label{suminest}
\sum\limits_{X/2<n\leq X}\min\left(1, \frac{1}{H\|n^c\|}\right)\ll X^\varepsilon\left(H^{-1}X+H^{\frac{1}{2}}X^{\frac{c}{2}}\right)\,.
\end{equation}
Let
\begin{equation*}
U=X^{\frac{1}{9}}\,,\quad V=X^{\frac{1}{3}}\,,\quad Z=\big[X^{\frac{4}{9}}\big]+\frac{1}{2}\,.
\end{equation*}
According to Lemma \ref{Heath-Brown}, the sum $S_0^\ast(t)$
can be decomposed into $O\Big(\log^{10}X\Big)$ sums, each of which is either of Type I
\begin{equation*}
\sum\limits_{|h|\leq H}c_h(t)\sum\limits_{m\sim M}a(m)\sum\limits_{l\sim L}e((h+t)m^cl^c)\,,
\end{equation*}
where
\begin{equation*}
L \gg Z\,, \quad  LM\asymp X\,, \quad |a(m)|\ll m^\varepsilon\,,
\end{equation*}
or of Type II
\begin{equation*}
\sum\limits_{|h|\leq H}c_h(t)\sum\limits_{m\sim M}a(m)\sum\limits_{l\sim L}b(l)e((h+t)m^cl^c)\,,
\end{equation*}
where
\begin{equation*}
U \ll L \ll V\,, \quad  LM\asymp X\,, \quad |a(m)|\ll m^\varepsilon\,,\quad |b(l)|\ll l^\varepsilon\,.
\end{equation*}
Using \eqref{Sast0}, Lemma \ref{SIest} and  Lemma \ref{SIIest} we obtain
\begin{equation}\label{Sast0est}
S_0^\ast(t)\ll X^{\frac{13993}{15276}+\varepsilon}\,.
\end{equation}
Taking into account \eqref{H}, \eqref{Lambdalog2}, \eqref{Sastformula}, \eqref{suminest}
and  \eqref{Sast0est} we establish the statement in the lemma.
\end{proof}

Bearing in mind \eqref{Gamma32}, Cauchy's inequality, Lemma \ref{intLintI}, Lemma \ref{IntK2}
and Lemma \ref{SalphaXest} we deduce
\begin{align}\label{Gamma32est}
\Gamma_3^{(2)}&\ll\max\limits_{\Delta\leq t\leq 1-\Delta}|S(t)|
\left(\int\limits_{\Delta}^{1-\Delta}|S(t)|^2\,dt\right)^{\frac{1}{2}}
\left(\int\limits_{\Delta}^{1-\Delta}|K(t)|^2\,dt\right)^{\frac{1}{2}}\nonumber\\
&\ll\max\limits_{\Delta\leq t\leq 1-\Delta}|S(t)|
\left(\int\limits_{0}^{1}|S(t)|^2\,dt\right)^{\frac{1}{2}}
\left(\int\limits_{0}^{1}|K(t)|^2\,dt\right)^{\frac{1}{2}}\nonumber\\
&\ll X^{\frac{29269}{15276}+\varepsilon}\ll\frac{ X^{3-c}}{\log X}\,.
\end{align}

\subsection{Estimation of $\mathbf{\Gamma_3}$}
\indent

Summarizing \eqref{Gamma3decomp}, \eqref{Gamma31est} and \eqref{Gamma32est} we obtain
\begin{equation}\label{Gamm3est}
\Gamma_3\ll\frac{ X^{3-c}}{\log X}\,.
\end{equation}

\section{Upper bound of $\mathbf{\Gamma_2}$}\label{SectionGamma2}
\indent

In this section we need a lemma that gives us information about the upper bound
of the number of solutions of the binary equation corresponding to \eqref{1}.
\begin{lemma}\label{Thenumberofsolutions}
Let $1<c<3$, $c\neq2$  and $N_0$ is a sufficiently large positive integer.
Then for the number of solutions $B_0(N_0)$ of the diophantine equation
\begin{equation}\label{Binary}
[p_1^c]+[p_2^c]=N_0
\end{equation}
in prime numbers $p_1,\,p_2 \in \left(N_0^{\frac{1}{c}}/2\,,\, N_0^{\frac{1}{c}}\right]$ we have that
\begin{equation*}
B_0(N_0)\ll \frac{ N_0^{\frac{2}{c}-1}}{\log^2N_0}\,.
\end{equation*}
\end{lemma}
\begin{proof}
Define
\begin{equation}\label{BX0}
B(X_0)=\sum\limits_{X_0/2<p_1,p_2\leq X_0\atop{[p_1^c]+[p_2^c]=N_0}}\log p_1\log p_2\,,
\end{equation}
where
\begin{equation}\label{X0}
X_0=N_0^{\frac{1}{c}}\,.
\end{equation}
By \eqref{BX0} we write
\begin{align}\label{BX0est1}
B(X_0)&=\int\limits_{-\Delta_0}^{1-\Delta_0} S_0^2(t) e(-N_0t)\,dt\nonumber\\
&=B_1(X_0)+B_2(X_0)\,,
\end{align}
where

\begin{align}
\label{S0t}
&S_0(t)=\sum\limits_{X_0/2<p\leq X_0} e(t [p^c])\log p\,,\\
\label{S0toverline}
&\overline{S}_0(t)=\sum\limits_{X_0/2<p\leq X_0} e(t p^c)\log p\,,\\
\label{Delta0}
&\Delta_0=\frac{(\log X_0)^{A_0}}{X_0^c}\,,\quad A_0>10\,,\\
\label{B1}
&B_1(X_0)=\int\limits_{-\Delta_0}^{\Delta_0}S_0^2(t)e(-N_0t)\,dt\,,\\
\label{B2}
&B_2(X_0)=\int\limits_{\Delta_0}^{1-\Delta_0}S_0^2(t)e(-N_0t)\,dt\,.
\end{align}
First we estimate $B_1(X_0)$.
Put
\begin{align}
\label{I0t}
&I_0(t)=\int\limits_{X_0/2}^{X_0}e(t y^c)\,dy\,,\\
\label{PsiDeltaX}
&\Psi_{\Delta_0}(X_0)=\int\limits_{-\Delta_0}^{\Delta_0}I_0^2(t)e(-N_0t)\,dt\,,\\
\label{PsiX}
&\Psi(X_0)=\int\limits_{-\infty}^{\infty}I_0^2(t)e(-N_0t)\,dt\,.
\end{align}
Using \eqref{IalphaXest}, \eqref{I0t} and \eqref{PsiX} we obtain
\begin{align}\label{Psiest}
\Psi(X_0)&=\int\limits_{-X_0^{-c}}^{X_0^{-c}}I_0^2(t)e(-N_0t)\,dt
+\int\limits_{|t|>X_0^{-c}}I_0^2(t)e(-N_0t)\,dt\,,\nonumber\\
&\ll\int\limits_{-X_0^{-c}}^{X_0^{-c}}X^2_0\,dt
+\int\limits_{X_0^{-c}}^{\infty} \left(\frac{X_0^{1-c}}{t}\right)^2\,dt\,,\nonumber\\
&\ll X_0^{2-c}\,.
\end{align}
On the other hand  \eqref{Sld-Soverline}, \eqref{S0t} -- \eqref{B1}, \eqref{PsiDeltaX},
Lemma \ref{SIasympt} and the trivial estimations
\begin{equation}\label{S0I0trivest}
S_0(t)\ll X_0 \,, \quad I_0(t)\ll X_0
\end{equation}
imply
\begin{align}\label{B1PsiDelta}
B_1(X_0)-\Psi_{\Delta_0}(X_0)
&\ll\int\limits_{-\Delta_0}^{\Delta_0}|S_0^2(t)-I_0^2(t)|\,dt\nonumber\\
&\ll\int\limits_{-\Delta_0}^{\Delta_0}
\big|S_0(t)-I_0(t)\big|\Big(|S_0(t)|+|I_0(t)|\Big)\,dt\nonumber\\
&\ll\Big(\max\limits_{|t|\leq\Delta_0}\big|\overline{S}_0(t)-I_0(t)\big|+\Delta_0 X_0\Big)
\left(\int\limits_{-\Delta_0}^{\Delta_0}|S_0(t)|\,dt
+\int\limits_{-\Delta_0}^{\Delta_0}|I_0(t)|\,dt\right)\nonumber\\
&\ll\Bigg(\frac{X_0}{e^{(\log X_0)^{\frac{1}{5}}}}+\Delta_0 X_0\Bigg)\Delta_0 X_0\nonumber\\
&\ll\frac{ X_0^{2-c}}{e^{(\log X_0)^{\frac{1}{6}}}}\,.
\end{align}
From \eqref{IalphaXest}, \eqref{Delta0}, \eqref{PsiDeltaX} and \eqref{PsiX} it follows
\begin{align}\label{PsiPsiDelta}
|\Psi(X_0)-\Psi_{\Delta_0}(X_0)|&\ll\int\limits_{\Delta_0}^{\infty}|I_0(t)|^2\,dt
\ll\frac{1}{X_0^{2(c-1)}}\int\limits_{\Delta_0}^{\infty}\frac{dt}{t^2}\nonumber\\
&\ll \frac{1}{X_0^{2(c-1)}\Delta_0}\ll\frac{X_0^{2-c}}{\log X_0}\,.
\end{align}
Now \eqref{Psiest}, \eqref{B1PsiDelta} and \eqref{PsiPsiDelta} and the identity
\begin{equation*}
B_1(X_0)=B_1(X_0)-\Psi_{\Delta_0}(X_0)+\Psi_{\Delta_0}(X_0)-\Psi(X_0)+\Psi(X_0)
\end{equation*}
give us
\begin{equation}\label{B1X0est}
B_1(X_0)\ll  X_0^{2-c}\,.
\end{equation}
Further we estimate $B_2(X_0)$.
By \eqref{X0}, \eqref{B2}, \eqref{S0I0trivest} and partial integration we deduce
\begin{align}\label{B2'est}
B_2(X_0)&=-\frac{1}{2\pi i}\int\limits_{\Delta_0}^{1-\Delta_0}\frac{S_0^2(t)}{N_0}\,d\,e(-N_0t)\nonumber\\
&=-\frac{S_0^2(t)e(-N_0t)}{2\pi iN_0}\Bigg|_{\Delta_0}^{1-\Delta_0}
+\frac{1}{2\pi iN_0}\int\limits_{\Delta_0}^{1-\Delta_0}e(-N_0t)\,d\Big(S_0^2(t)\Big)\nonumber\\
&\ll X_0^{2-c}+X_0^{-c}|\Omega|\,,
\end{align}
where
\begin{equation}\label{Omega}
\Omega=\int\limits_{\Delta_0}^{1-\Delta_0}e(-N_0t)\,d\Big(S_0^2(t)\Big)\,.
\end{equation}
Next we consider $\Omega$. Put
\begin{equation}\label{Gammat}
\Gamma \, :\, z=f(t)=S_0^2(t)\,,\quad \Delta_0\leq t\leq 1-\Delta_0\,.
\end{equation}
Now \eqref{Omega} and \eqref{Gammat} imply
\begin{equation}\label{Omegaest1}
\Omega=\int\limits_{\Gamma} e\Big(-N_0f^{-1}(z)\Big)\,dz\,.
\end{equation}
Using \eqref{S0I0trivest}, \eqref{Gammat} and that the integral \eqref{Omegaest1} is independent of path we derive
\begin{equation}\label{Omegaest2}
\Omega=\int\limits_{\overline{\Gamma}} e\Big(-N_0f^{-1}(z)\Big)\,dz\ll\int\limits_{\overline{\Gamma}} |dz|
\ll |f(\Delta_0)|+|f(1-\Delta_0)| \ll X_0^2\,,
\end{equation}
where $\overline{\Gamma}$ is the line segment connecting the points $f(\Delta_0)$ and $f(1-\Delta_0)$.
Bearing in mind \eqref{B2'est} and \eqref{Omegaest2} we  find
\begin{equation}\label{B2X0est}
B_2(X_0)\ll X_0^{2-c}\,.
\end{equation}
Summarizing \eqref{BX0est1}, \eqref{B1X0est} and \eqref{B2X0est} we obtain
\begin{equation}\label{BX0est}
B(X_0)\ll X_0^{2-c}\,.
\end{equation}
Taking into account \eqref{BX0}, \eqref{X0} and \eqref{BX0est}, for the number of solutions $B_0(N_0)$
of the  diophantine equation \eqref{Binary} we get
\begin{equation*}
B_0(N_0)\ll \label{B1Psitau}\frac{N_0^{\frac{2}{c}-1}}{\log^2N_0}\,.
\end{equation*}
The lemma is proved.
\end{proof}

We are now ready to  estimate the sum $\Gamma_2$.
We denote by $\mathcal{F}(X)$ the set of all primes
$X/2<p\leq X$ such that $p-1$ has a divisor belongs to the interval $(D,X/D)$.
The inequality $xy\leq x^2+y^2$ and  \eqref{Gamma2} give us
\begin{align*}
\Gamma_2^2&\ll(\log X)^6\sum\limits_{X/2<p_1,...,p_6\leq X
\atop{[p^c_1]+[p^c_2]+[p^c_3]=N
\atop{[p^c_4]+[p^c_5]+[p^c_6]=N}}}
\left|\sum\limits_{d|p_1-1\atop{D<d<X/D}}\chi_4(d)\right|
\left|\sum\limits_{t|p_4-1\atop{D<t<X/D}}\chi_4(t)\right|\\
&\ll(\log X)^6\sum\limits_{X/2<p_1,...,p_6\leq X
\atop{[p^c_1]+[p^c_2]+[p^c_3]=N
\atop{[p^c_4]+[p^c_5]+[p^c_6]=N
\atop{p_4\in\mathcal{F}(X)}}}}\left|\sum\limits_{d|p_1-1\atop{D<d<X/D}}\chi_4(d)\right|^2\,.
\end{align*}
The summands in the last sum for which $p_1=p_4$ can be estimated with
$\mathcal{O}\big(X^{3+\varepsilon}\big)$.\\
Thus
\begin{equation}\label{Gamma2est1}
\Gamma_2^2\ll(\log X)^6\Sigma_0+X^{3+\varepsilon}\,,
\end{equation}
where
\begin{equation}\label{Sigma0}
\Sigma_0=\sum\limits_{X/2<p_1\leq X}
\left|\sum\limits_{d|p_1-1\atop{D<d<X/D}}\chi_4(d)\right|^2
\sum\limits_{X/2<p_4\leq X\atop{p_4\in\mathcal{F}(X)
\atop{p_4\neq p_1}}}\sum\limits_{X/2<p_2, p_3, p_5, p_6\leq X
\atop{[p^c_1]+[p^c_2]+[p^c_3]=N
\atop{[p^c_4]+[p^c_5]+[p^c_6]=N}}}1\,.
\end{equation}
Now  \eqref{Sigma0} and Lemma \ref{Thenumberofsolutions} yield
\begin{equation}\label{Sigma0est}
\Sigma_0\ll \frac{X^{4-2c}}{\log^4X}\,\Sigma'_0\,\Sigma''_0\,,
\end{equation}
where
\begin{equation*}
\Sigma'_0=\sum\limits_{X/2<p\leq X}\left|\sum\limits_{d|p-1\atop{D<d<X/D}}\chi_4(d)\right|^2\,,
\quad \Sigma''_0=\sum\limits_{X/2<p\leq X\atop{p\in\mathcal{F}(X)}}1\,.
\end{equation*}
Applying Lemma \ref{Hooley1} we obtain
\begin{equation}\label{Sigma0'est}
\Sigma'_0\ll\frac{X(\log\log X)^7}{\log X}\,.
\end{equation}
Using Lemma \ref{Hooley2} we get
\begin{equation}\label{Sigma0''est}
\Sigma''_0\ll\frac{X(\log\log X)^3}{(\log X)^{1+2\theta_0}}\,,
\end{equation}
where $\theta_0$ is denoted by  \eqref{theta0}.

Finally \eqref{Gamma2est1}, \eqref{Sigma0est},
\eqref{Sigma0'est} and \eqref{Sigma0''est} imply
\begin{equation}\label{Gamma2est2}
\Gamma_2\ll\frac{ X^{3-c}(\log\log X)^5}{(\log X)^{\theta_0}}\,.
\end{equation}

\section{Asymptotic formula for $\mathbf{\Gamma_1}$}\label{SectionGamma1}
\indent

Consider the sum $\Gamma_1$.
From \eqref{Gamma1}, \eqref{Ild} and \eqref{Ilddecomp} we deduce
\begin{equation}\label{Gamma1decomp}
\Gamma_1=\Gamma_1^{(1)}+\Gamma_1^{(2)}\,,
\end{equation}
where
\begin{equation}\label{Gamma1i}
\Gamma_1^{(i)}=\sum\limits_{d\leq D}\chi_4(d)I_{1,d}^{(i)}(N)\,,\;\; i=1,\,2.
\end{equation}

\subsection{Estimation of $\mathbf{\Gamma_1^{(1)}}$}
\indent

First we consider $\Gamma_1^{(1)}$.
Using formula \eqref{Ild1est} for $J=(X/2,X]$, \eqref{Gamma1i}
and treating the reminder term by the same way as for $\Gamma_3^{(1)}$
we find
\begin{equation}\label{Gamma11est1}
\Gamma_1^{(1)}=\Phi(X)\sum\limits_{d\leq D}\frac{\chi_4(d)}{\varphi(d)}
+\mathcal{O}\bigg(\frac{X^{3-c}}{\log X}\bigg)\,,
\end{equation}
where
\begin{equation}\label{PhiX}
\Phi(X)=\int\limits_{-\infty}^{\infty}I^3(t)e(-Nt)\,dt\,.
\end{equation}

\begin{lemma}\label{IIIest} For the integral denoted by \eqref{PhiX} the asymptotic formula
\begin{equation*}
\Phi(X)=\frac{\Gamma^3\left(1+\frac{1}{c}\right)}{\Gamma\left(\frac{3}{c}\right)}\left(1-\frac{1}{2^{3-c}}\right)X^{3-c}
+\mathcal{O}\Bigg(\frac{ X^{3-c}}{e^{(\log X)^{\frac{1}{6}}}}\Bigg)
\end{equation*}
holds.
\end{lemma}

\begin{proof}
From \eqref{PhiX} write
\begin{equation}\label{IIIdecomp}
\Phi(X)= \Theta_1-\Omega_S+\Omega_S+\Theta_2\,,
\end{equation}
where \begin{align}
\label{Theta1}
&\Theta_1=\int\limits_{-\Delta}^\Delta I^3(t) e(-Nt)\,dt\,,\\
\label{Theta2}
&\Theta_2=\int\limits_{|t|> \Delta} I^3(t) e(-Nt)\,dt\,,\\ \label{OmegaS}
&\Omega_S= \int\limits_{-\Delta}^\Delta S^3(t) e(-Nt)\,dt\,.
\end{align}

By \eqref{Delta}, \eqref{SalphaX}, \eqref{Sldoverline}, \eqref{Sld-Soverline}, \eqref{Theta1},
\eqref{OmegaS}, Lemma \ref{SIasympt} and Lemma \ref{intLintI}  we obtain
\begin{align}\label{Theta1OmegaS}
\Theta_1-\Omega_S
&\ll\int\limits_{-\Delta}^{\Delta}\big|S^3(t)-I^3(t)\big|\,dt\nonumber\\
&\ll\int\limits_{-\Delta}^{\Delta}
\big|S(t)-I(t)\big|\Big(\big|S^2(t)-I^2(t)\big|\Big)\,dt\nonumber\\
&\ll\Big(\max\limits_{|t|\leq\Delta}\big|\overline{S}(t)-I(t)\big|+\Delta X\Big)
\left(\int\limits_{-\Delta}^{\Delta}|S(t)|^2\,dt
+\int\limits_{-\Delta}^{\Delta}|I(t)|^2\,dt\right)\nonumber\\
&\ll\Bigg(\frac{X}{e^{(\log X)^{\frac{1}{5}}}}+\Delta X\Bigg) X^{2-c}\log^2X\nonumber\\
&\ll\frac{ X^{3-c}}{e^{(\log X)^{\frac{1}{6}}}}\,.
\end{align}
Arguing as in  \cite{Laporta-Tolev}  we get
\begin{equation}\label{OmegaSformula}
\Omega_S =\frac{\Gamma^3\left(1+\frac{1}{c}\right)}{\Gamma\left(\frac{3}{c}\right)}\left(1-\frac{1}{2^{3-c}}\right)X^{3-c}
+\mathcal{O}\Bigg(\frac{ X^{3-c}}{e^{(\log X)^{\frac{1}{3}-\varepsilon}}}\Bigg) \,.
\end{equation} From \eqref{Delta}, \eqref{IalphaXest} and \eqref{Theta2} it follows
\begin{align}\label{Theta2est}
\Theta_2&\ll\int\limits_{\Delta}^{\infty}|I(t)|^3\,dt
\ll\frac{1}{X^{3(c-1)}}\int\limits_{\Delta}^{\infty}\frac{dt}{t^3}\nonumber\\
&\ll \frac{1}{X^{3(c-1)}\Delta^2}\ll X^{3-c-\varepsilon}\,.
\end{align}
Now the lemma follows from  \eqref{IIIdecomp}, \eqref{Theta1OmegaS}, \eqref{OmegaSformula}  and \eqref{Theta2est}.
\end{proof}
According to \cite{Dimitrov3} we have
\begin{equation}\label{sumchiphi}
\sum\limits_{d\leq D}\frac{\chi_4(d)}{\varphi(d)}=
\frac{\pi}{4}\prod\limits_p \left(1+\frac{\chi_4(p)}{p(p-1)}\right)+\mathcal{O}\Big(X^{-1/20}\Big)\,.
\end{equation}
From \eqref{Gamma11est1} and \eqref{sumchiphi} we obtain
\begin{equation}\label{Gamma11est2}
\Gamma_1^{(1)}=\frac{\pi}{4}\prod\limits_p \left(1+\frac{\chi_4(p)}{p(p-1)}\right) \Phi(X)
+\mathcal{O}\bigg(\frac{ X^{3-c}}{\log X}\bigg)+\mathcal{O}\Big(\Phi(X)X^{-1/20}\Big)\,.
\end{equation}
Now \eqref{Gamma11est2} and Lemma \ref{IIIest} yield
\begin{equation}\label{Gamma11est3}
\Gamma_1^{(1)}=\frac{\pi}{4}\prod\limits_p \left(1+\frac{\chi_4(p)}{p(p-1)}\right)
\frac{\Gamma^3\left(1+\frac{1}{c}\right)}{\Gamma\left(\frac{3}{c}\right)}\left(1-\frac{1}{2^{3-c}}\right)X^{3-c}
+\mathcal{O}\bigg(\frac{ X^{3-c}}{\log X}\bigg)\,.
\end{equation}

\subsection{Estimation of $\mathbf{\Gamma_1^{(2)}}$}
\indent

Arguing as in the estimation of $\Gamma_3^{(2)}$ we get
\begin{equation} \label{Gamma12est}
\Gamma_1^{(2)}\ll\frac{ X^{3-c}}{\log X}\,.
\end{equation}

\subsection{Estimation of $\mathbf{\Gamma_1}$}
\indent

Summarizing  \eqref{Gamma1decomp}, \eqref{Gamma11est3} and \eqref{Gamma12est}  we deduce
\begin{equation}\label{Gamma1est}
\Gamma_1=\frac{\pi}{4}\prod\limits_p \left(1+\frac{\chi_4(p)}{p(p-1)}\right)
\frac{\Gamma^3\left(1+\frac{1}{c}\right)}{\Gamma\left(\frac{3}{c}\right)}\left(1-\frac{1}{2^{3-c}}\right)X^{3-c}
+\mathcal{O}\bigg(\frac{ X^{3-c}}{\log X}\bigg)\,.
\end{equation}

\section{Proof of the Theorem}\label{Sectionfinal}
\indent

Bearing in mind  \eqref{X}, \eqref{Gamma0decomp}, \eqref{Gamm3est}, \eqref{Gamma2est2} and \eqref{Gamma1est}
we establish  asymptotic formula \eqref{asymptoticformula1}.

\vskip18pt
\footnotesize
\begin{flushleft}
S. I. Dimitrov\\
Faculty of Applied Mathematics and Informatics\\
Technical University of Sofia \\
8, St.Kliment Ohridski Blvd. \\
1756 Sofia, BULGARIA\\
e-mail: sdimitrov@tu-sofia.bg\\
\end{flushleft}

\end{document}